\documentclass[10pt]{amsart}
\usepackage{indentfirst}
\usepackage{amsmath,amssymb,latexsym,esint,cite,mathrsfs}
\usepackage{verbatim,wasysym}
\usepackage[left=2.6cm,right=2.6cm,top=2.9cm,bottom=2.7cm]{geometry}
\usepackage{tikz,enumitem,graphicx, subfig, microtype, color}
\usepackage{epic,eepic}
\usepackage[colorlinks=true,urlcolor=blue, citecolor=red,linkcolor=blue,
linktocpage,pdfpagelabels, bookmarksnumbered,bookmarksopen]{hyperref}
\usepackage[hyperpageref]{backref}
\usepackage[english]{babel}
\numberwithin{equation}{section}
\newtheorem{thm}{Theorem}[section]
\newtheorem{lem}[thm]{Lemma}

\newtheorem{Prop}[thm]{Proposition}

\begin{document}
\title[Helmholtz equation]{On concentration of real solutions for fractional Helmholtz equation}
\author[Z.\ Shen]{Zifei Shen}
\author[S.\ Zhang]{Shuijin Zhang}
\address{Zifei Shen, Shuijin Zhang \newline\indent Department of Mathematics, Zhejiang Normal University, \newline\indent
	Jinhua, Zhejiang, 321004, People's Republic of China}
\email{Z. Shen: szf@zjnu.edu.cn; S. Zhang: shuijinzhang@zjnu.edu.cn}

.

\begin{abstract}
This paper studies the nonlinear fractional Helmholtz equation
\begin{equation}\label{main}
(-\Delta)^{s} u-k^{2} u=Q(x)|u|^{p-2}u, ~~\mathrm{in}~~\mathbb{R}^{N},~~N\geq3,
\end{equation} 
where $\frac{N}{N+1}<s<\frac{N}{2}$, $\frac{2(N+1)}{N-1}<p<\frac{2N}{N-2s}$ are two real exponents, and the coefficient $Q$ is bounded continuous, nonnegative and satisfies the condition
\begin{equation}
\mathop{\mathrm{lim~sup}}\limits_{|x|\longrightarrow\infty}Q(x)
<\mathop{\mathrm{sup}}\limits_{x\in\mathbb{R}^{N}}Q(x).
\end{equation}
For $k>0$ large, the existence of real-valued solutions for (\ref{main}) are proved, and in the limit $k\longrightarrow\infty$, sequence of solutions associated with ground states of a dual equation are shown to concentrate, after rescaling, at global maximum points of the function $Q$.
\end{abstract}     
\maketitle
%
%\medskip
\setlength{\parindent}{2em}
\section{Introduction and Main Results}
In this paper, we are concerned with the nonlinear fractional Helmholtz equation
\begin{equation}\label{main1}
(-\Delta)^{s} u-k^{2} u=Q(x)|u|^{p-2}u, ~~\mathrm{in}~~\mathbb{R}^{N},~~N\geq3,
\end{equation}
where $\frac{N}{N+1}<s<\frac{N}{2}$, $\frac{2(N+1)}{N-1}<p<\frac{2N}{N-2s}$ are two real exponents, $Q$ is a bounded continuous function.

When $s=1$, the Helmholtz equation
\begin{equation}\label{sin}
-\Delta u-k^{2}u=f(x,u)~~~\mathrm{in}~\mathbb{R}^{N},
\end{equation}
has attracted immense attention in the recent decades due to the importance in Scattering and Optics. The main feature of this problem is that the parameter $k^{2}>0$ is contained in the essential spectrum of negative Laplace $-\Delta$. A general method to detect the existence of weak solutions is the Linking argument, that is find the critical point of the corresponding functional. However, one can not find a appropriate space in where the associated functional of (\ref{sin}) can be well defined, this is due to that the oscillating solutions with slow decay which, in general, are not elements of $H^{1}(\mathbb{R}^{N})$. Therefore, the direct variational approach is invalid.

To overcome this difficult, a dual invariant method has been proposed, which is based on the ``Limiting Absorption Principle''. By constructing the auxiliary problems
\begin{equation}\label{auxiliary}
-\Delta u-(\lambda+i\varepsilon)u=f(x,u)~~~\mathrm{in}~~\mathbb{R}^{N},
\end{equation}
one can obtain the boundedness estimate for the resolvent operator \begin{equation*}
\mathcal{R}_{\lambda,\varepsilon}=(-\Delta-(\lambda+i\varepsilon))^{-1},
\end{equation*}
and as $\varepsilon\longrightarrow 0^{+}$, one can also obtain the boundedness estimate for the resolvent $\mathcal{R}_{\lambda}=(-\Delta-\lambda)^{-1}$, see \cite[Theorem 6]{Gutierrez2004} or see \cite{Mandel2019-1,Odeh1961,Radosz2010,Cossetti2021,Mandel2021}. Based on the boundedness estimate, Ev\'{e}quoz and Weth \cite{Evequoz2015} (see \cite{Evequoz2016}) set up a dual variational framework for (\ref{sin}). Correspondingly, the nontrivial real-valued solutions of equation (\ref{sin}) with $f(x,u)=Q(x)|u|^{p-2}u$ are detected via the mountain pass argument, where $Q(x)$ is a positive weight function, see also \cite{Evequoz2017-2,Evequoz2015-1,Mandel2021-1,Evequoz2020,Mandel2019} for the other cases. By a similar way, Shen and the second author \cite{Yang2023} also obtained the real valued solutions for the fractional Helmholtz equation (\ref{main1}) in the case that $0<k^{2}<+\infty$ and $Q(x)$ is assumed to be a periodic or decay function.

Recently, Ev\'{e}quoz \cite{Evequoz2017-1} considered (\ref{sin}) in a limit case and obtain some surprising results on the solutions. If $Q$ is assumed to be a bounded continuous function,
equation (\ref{sin}) still possess a real valued solution for $k$ large enough. Furthermore, the solutions concentrate at the global maximum points of the function $Q(x)$ as the frequency $\lambda$ tends infinity. Actually, the concentrating solutions is also a big and time honored topic for the nonlinear Schr\"{o}dinger equation
\begin{equation}\label{schrodinger}
-\varepsilon^{2}\Delta u+V(x)u=Q(x)|u|^{p-2}u~~~\mathrm{in}~\mathbb{R}^{N},
\end{equation}
where $V(x)\geq0$ is a potential function. Basically, there are two main routes have been pursued to investigate the concentrating solutions. One is the Lyapunov-Schmidt reduction scheme proposed by Floer and Weinstein \cite{Floer1986}, which has been further extended and combined with variational arguments by Ambrosetti et al. \cite{Ambrosetti1997, Ambrosetti2001,Ambrosetti2003,Ambrosetti2004}, see also for example \cite{Li1997,Pistoia2002} for multibump solutions. Another one is the purely variational approach initiated by Rabinowitz \cite{Rabinowitz1986}, which is mainly relayed by del Pino and Felmer \cite{DelPino1996,DelPino1997,DelPino1998,DelPino2002}. More precisely, under the global condition
\begin{equation}\label{V}
\mathop{\mathrm{lim~inf}}\limits_{|x|\longrightarrow\infty}~V(x)>\mathop{\mathrm{inf}}\limits_{x\in\mathbb{R}^{N}}~V(x),
\end{equation}
it was proved in \cite{Rabinowitz1991} that a ground state (i.e., positive least-energy solution) of (\ref{schrodinger}) exists for small $\varepsilon>0$. In the limit $\varepsilon\longrightarrow0$, Wang \cite{Wang1993} showed that sequences of ground states concentrate at a global minimum point $x_{0}$ of $V$ and converge, after rescaling, toward the ground state of the limit problem
\begin{equation}\label{limit1}
-\Delta u+V(x_{0})u=|u|^{p-2}u~~~\mathrm{in}~\mathbb{R}^{N}.
\end{equation}
These results are also extended to the fractional Shr\"{o}dinger equation, that is
\begin{equation}
\varepsilon^{2s}(-\Delta)^{s}u+V(x)u=f(x,u),~~~\mathrm{in}~\mathbb{R}^{N}.
\end{equation}
We would refer the papers \cite{Davila2014,Chen2014} and \cite{Alves2015,He2016} to readers.

Motivated by these works, we also consider in this paper that the existence and concentrating phenomenon of the solutions of (\ref{main1}) as $k\longrightarrow\infty$. However, as we introduced before, the structure of the Helmholtz equation is vary complex. There is no uniqueness and nondegeneracy result for the real valued solutions, therefore, the classical methods of reduction may be not available to construct the concentrating solutions. This impels us to consider the another method what proposed by Rabinowitz. Actually, the variational method also can not be directly adapt in our case, since there is no natural concept of ground state associated the direct variational formulation.

Follow the idea in \cite{Evequoz2017-1}, we define the dual ground state for (\ref{main1}) as follow. Setting $\varepsilon=k^{-1}$, $u_{\varepsilon}(x)=\varepsilon^{\frac{2s}{p-2}}u(\varepsilon x)$ and $Q_{\varepsilon}(x)=Q(\varepsilon x)$, $x\in\mathbb{R}^{N}$, (\ref{main1}) can be rewritten as
\begin{equation}\label{rescaling}
(-\Delta)^{s} u_{\varepsilon}-u_{\varepsilon}=Q_{\varepsilon}(x)|u_{\varepsilon}|^{p-2}u_{\varepsilon}~~~\mathrm{in}~\mathbb{R}^{N}.
\end{equation}
Furthermore, setting $v=Q^{\frac{1}{p'}}|u_{\varepsilon}|^{p-2}u_{\varepsilon}$, we are led to consider the integral equation
\begin{equation}
|v|^{p-2}v=Q_{\varepsilon}^{\frac{1}{p}}[\mathbf{R}^{s}\ast(Q_{\varepsilon}^{\frac{1}{p}}v)]
\end{equation}
where $p'=\frac{p}{p-1}$ and $\mathbf{R}^{s}$ denotes the real part of the fractional Helmholtz resolvent operator, see \cite{Shen2023}. The solutions of this integral equation are critical points of the so-called dual energy functional $J_{\varepsilon}:L^{p'}(\mathbb{R}^{N})\longrightarrow \mathbb{R}$ given by
\begin{equation}\label{J varepsilon}
J_{\varepsilon}(v)=\frac{1}{p'}\int_{\mathbb{R}^{N}}|v|^{p'}\mathrm{d}x-\frac{1}{2}\int_{\mathbb{R}^{N}}Q_{\varepsilon}^{\frac{1}{p}}v\mathbf{R}^{s}\big(Q_{\varepsilon}^{\frac{1}{p}}v\big)\mathrm{d}x.
\end{equation}
Furthermore, every critical point $v$ of $J_{\varepsilon}$ gives rise to a strong solution $u$ of (\ref{main1}) with $k=\frac{1}{\varepsilon}$, by setting
\begin{equation}\label{solution}
u(x)=k^{\frac{2s}{p-2}}\mathbf{R}^{s}\big(Q_{\varepsilon}^{\frac{1}{p}}v\big)(kx),~~~x\in\mathbb{R}^{N}.
\end{equation}
This correspondence allows us to define a notion of ground state for (\ref{main1}) as follows. If $\varepsilon=\frac{1}{k}$, and $v$ is a nontrivial critical point for $J_{\varepsilon}$ at the mountain pass level, the function $u$ given by (\ref{solution}) will be called a $dual~~ground~~state$ of (\ref{main1}).

Apparently, once we obtain the existence and concentration of $v$, we then obtain the existence and concentration of $u$, up to rescaling, of sequences of dual ground states. There we present our first main result.
\begin{thm}
Let $N\geq3$, $\frac{N}{N+1}<s<\frac{N}{2}$, $\frac{2(N+1)}{N-1}<p<\frac{2N}{N-2s}$ and consider a function $Q$ satisfying

(Q0)~Q is continuous, bounded and $Q\geq 0$ on $\mathbb{R}^{N}$;

(Q1)~$Q_{\infty}:=\mathop{\mathrm{lim~sup}}\limits_{|x|\longrightarrow\infty}~Q(x)<Q_{0}:=\mathop{\mathrm{sup}}\limits_{x\in\mathbb{R}^{N}}~Q(x)$.

\noindent (i)~There is $k_{0}>0$ such that for all $k>k_{0}$, the problem (\ref{main1}) admits a dual ground state.

\noindent (ii)~Let $(k_{n})_{n}\subset(k_{0},\infty)$ satisfy $\mathop{\mathrm{lim}}\limits_{n\longrightarrow\infty}k_{n}=\infty$ and consider for each $n$, a dual ground state $u_{n}$ of
\begin{equation}
(-\Delta)^{s} u-k_{n}^{2}u=Q(x)|u|^{p-2}u~~~\mathrm{in}~\mathbb{R}^{N}.
\end{equation}
Then there is a maximum point $x_{0}$ of $Q$, a dual ground state $u_{0}$ of
\begin{equation}\label{limit equation}
(-\Delta)^{s} u-u=Q_{0}|u|^{p-2}u~~~\mathrm{in}~\mathbb{R}^{N}
\end{equation}
and a sequence $(x_{n})_{n}\subset\mathbb{R}^{N}$ such that (up to a subsequence) $\mathop{\mathrm{lim}}\limits_{n\longrightarrow\infty}x_{n}=x_{0}$ and
\begin{equation}
k_{n}^{-\frac{2}{p-2}}u_{n}\big(\frac{\cdot}{k_{n}}+x_{n}\big)\longrightarrow u_{0}~~~\mathrm{in}~L^{p}(\mathbb{R}^{N}),~~\mathrm{as}~n\longrightarrow\infty.
\end{equation}
\end{thm}
%For the Schr\"{o}dinger equation (\ref{schrodinger}) with $V\equiv1$, Wang and Zeng\cite{Wang1997} %noticed that (\ref{Q}) plays the same role as the Rabinowitz condition (\ref{V}). As a consequence of %Theorem 1.1, we see that this condition also ensures the concentration, in the $L^{p}$-sense, for %(\ref{schrodinger}) with $V\equiv-1$. To the best of our knowledge, this is the first concentration %result for semilinear problems where $0$ lies in the interior of the essential spectrum of the %linearization.
Due to the assumption on $Q$, we show that, for some $\varepsilon$ small enough, the dual energy functional is strictly below the least among all possible energy levels for the problem at infinity, see Section 2. Correspondingly, we prove that all the dual energy functional satisfies the Palais-Smale condition, and therefore prove the first assertion of the above theorem, see Section 3. The proof of the second assertion depends on a representation lemma and we finish it in Section 4.

Follow the same idea in \cite[Theorem 2]{Evequoz2017-1}, we can also obtain the multiplicity result for (\ref{main1}). Let $M=\{x\in\mathbb{R}^{N}:Q(x)=Q_{0}\}$ denotes the set of maximum points of $Q$, and for $\delta>0$ we let $M_{\delta}=\{x\in\mathbb{R}^{N}:\mathrm{dist}(x,M)\leq\delta\}$. Also, for a closed subset $Y$ of a metric space $X$ we denote by $\mathrm{cat_{X}}(Y)$ the Lusternik-Schnirelmann category of $Y$ with respect to $X$, i.e., the least number of closed contractible sets in $X$ which cover $Y$.
\begin{thm}\label{thn2}
Let $N\geq3$, $\frac{N}{N+1}<s<\frac{N}{2}$, $\frac{2(N+1)}{N-1}<p<\frac{2N}{N-2s}$ and consider a function $Q$ satisfying (Q0) and (Q1). For every $\delta>0$, there exists $k(\delta)>0$ such that $(\ref{main1})$ has at least $\mathrm{cat}_{M_{\delta}}(M)$ nontrivial solutions for all $k>k_{\delta}$.
\end{thm}

The proof of the above result depends on a topological arguments close to the
ones developed by Cingolani and Lazzo \cite{Cingolani1997} for (\ref{schrodinger}) (see also \cite{Cingolani2000}) and based on ideas of Benci, Cerami and Passaseo \cite{Benci1991,Benci1991-1} for problems on bounded domains. The main point lies in the construction of two maps whose composition is homotopic to the inclusion $M\hookrightarrow M_{\delta}$, we omit it here and refer readers to \cite[Theorem 2]{Evequoz2017-1}.

We close the introduction by some notations. For $1\leq q\leq\infty$, we write $||\cdot||_{q}$ instead $||\cdot||_{L^{q}(\mathbb{R}^{N})}$ for the standard norm of the Lebesgue space $L^{q}(\mathbb{R}^{N})$. In addition, for $r>0$ and $x\in\mathbb{R}^{N}$, we denote by $B_{r}(x)$ the open ball in $\mathbb{R}^{N}$ of radius $r$ centered at $x$, and let $B_{r}=B_{r}(0)$.

\section{The variational framework}
\subsection{Dual Functional}

Before we compare the energy functional, we recall some properties of the dual functional (\ref{J varepsilon}). Since $p'<2$ and since the kernel of the operator $\mathbf{R}^{s}$ is positive close to the origin, the geometry of the functional $J_{\varepsilon}$ is of mountain pass type:
\begin{equation}\label{mountain1}
\exists \alpha>0~~\mathrm{and}~~~ \rho>0~~~\mathrm{such~~that}~~ J_{\varepsilon}(v)\geq\alpha>0,~~\forall~v\in L^{p'}(\mathbb{R}^{N})~~\mathrm{with}~~||v||_{p'}=\rho.
\end{equation}

\begin{equation}\label{mountain2}
\exists v_{0}\in L^{p'}(\mathbb{R}^{N})~~\mathrm{such~~that}~~ ||v_{0}||_{p'}>\rho~~\mathrm{and} ~~J_{\varepsilon}(v_{0})<0.
\end{equation}
As a consequence, the Nehari set associated to $J_{\varepsilon}$:
\begin{equation}
\mathcal{N}_{\varepsilon}:=\{v\in L^{p'}(\mathbb{R}^{N})\setminus\{0\}:J'_{\varepsilon}(v)v=0\},
\end{equation}
is not empty. More precisely, by (\ref{mountain1}), the set
\begin{equation}
U^{+}_{\varepsilon}:=\{v\in L^{p'}(\mathbb{R}^{N}):\int_{\mathbb{R}^{N}}Q_{\varepsilon}^{\frac{1}{p}}v\mathbf{R}^{s}\big(Q_{\varepsilon}^{\frac{1}{p}}v\big)\mathrm{d}x>0\}
\end{equation}
is not empty and for each $v\in U^{+}_{\varepsilon}$ there is a unique $t_{v}>0$ such that $t_{v}v\in\mathcal{N}_{\varepsilon}$ holds. It is given by
\begin{equation}\label{t}
t_{v}^{2-p'}=\frac{\int_{\mathbb{R}^{N}}|v|^{p'}\mathrm{d}x}{\int_{\mathbb{R}^{N}}Q_{\varepsilon}^{\frac{1}{p}}v\mathbf{R}^{s}\big(Q_{\varepsilon}^{\frac{1}{p}}v\big)\mathrm{d}x}.
\end{equation}
In addition, $t_{v}$ is the unique maximum point of $t\mapsto J_{\varepsilon}(tv)$, $t\geq0$. Using (\ref{mountain1}), we obtain in particular
\begin{equation} 
c_{\varepsilon}:=\mathop{\mathrm{inf}}\limits_{\mathcal{N}_{\varepsilon}}~J_{\varepsilon}=\mathop{\mathrm{inf}}\limits_{v\in U^{+}_{\varepsilon}}~J_{\varepsilon}(t_{v}v)>0.
\end{equation}
Moreover, for every $v\in\mathcal{N}_{\varepsilon}$ we have $c_{\varepsilon}\leq J_{\varepsilon}(v)=(\frac{1}{p'}-\frac{1}{2})||v||_{p'}^{p'}$. Hence, $0$ is isolated in the set $\{v\in L^{p'}(\mathbb{R}^{N}):~J'_{\varepsilon}(v_{n})\longrightarrow 0\}$ and as a consequence, the $C^{1}$-submanifold $\mathcal{N}_{\varepsilon}$ of $L^{p'}(\mathbb{R}^{N})$ is complete.

We recall that $(v_{n})_{n}\subset L^{p'}(\mathbb{R}^{N})$ is termed a Palais-Smale sequence, or a (PS)-sequence, for $J_{\varepsilon}$ if $(J_{\varepsilon}(v_{n}))_{n}$ is bounded and $J'_{\varepsilon}(v_{n})\longrightarrow0$ as $n\longrightarrow\infty$. Also, for $d>0$, we say that $(v_{n})_{n}$ is a $(\mathrm{PS})_{d}$-sequence for $J_{\varepsilon}$ if it is a $(\mathrm{PS})$-sequence and if $J_{\varepsilon}\longrightarrow d$ as $n\longrightarrow\infty$. The following properties hold (see \cite[Sect.2]{Shen2023}).

\begin{lem}\label{PS}
Let $(v_{n})_{n}\subset L^{p'}(\mathbb{R}^{N})$ be a Palais-Smale sequence for $J_{\varepsilon}$. Then $(v_{n})_{n}$ is bounded and there exists $v\in L^{p'}(\mathbb{R}^{N})$ such that $J'_{\varepsilon}(v)=0$ and, up to a subsequence, $v_{n}\rightharpoonup v$ weakly in $L^{p'}(\mathbb{R}^{N})$ and $J_{\varepsilon}(v)\leq\mathop{\mathrm{lim~inf}}\limits_{n\longrightarrow\infty}~J_{\varepsilon}(v_{n})$.

Moreover, for every bounded and measurable set $B\subset\mathbb{R}^{N}$, $1_{B}v_{n}\longrightarrow 1_{B}v$ strongly in $L^{p'}(\mathbb{R}^{N})$.
\end{lem}

As a consequence, we obtain the following characterization of the infimum $c_{\varepsilon}$ of $J_{\varepsilon}$ over the Nehari manifold $\mathcal{N}_{\varepsilon}$ (see \cite[Sect.4]{Shen2023}).
\begin{lem}\label{mountain3}
(i)~$c_{\varepsilon}$ coincides with the mountain pass level, i.e.,
\begin{equation*}
c_{\varepsilon}=\mathop{\mathrm{inf}}\limits_{\gamma\in\Gamma}\mathop{\mathrm{max}}\limits_{t\in[0,1]}J_{\varepsilon}(\gamma(t)),~~\mathrm{where}
\end{equation*}
\begin{equation*}
\Gamma=\{\gamma\in C([0,1], L^{p'}(\mathbb{R}^{N})):\gamma(0)=0,~\mathrm{and}~J(\gamma(1))<0\}.
\end{equation*}

(ii)~If $c_{\varepsilon}$ is attained, then $c_{\varepsilon}=\mathrm{min}\{J_{\varepsilon}(v):v\in L^{p'}(\mathbb{R}^{N})\setminus\{0\}, ~J'_{\varepsilon}(v)=0\}$.

(iii)~If $Q_{\varepsilon}$ is constant or $\mathbb{Z}^{N}-$periodic, then $c_{\varepsilon}$ is attained.
\end{lem}

\subsection{Energy Compare}

Consider the functional
\begin{equation}
J_{0}(v):=\frac{1}{p'}\int_{\mathbb{R}^{N}}|v|^{p'}dx-\frac{1}{2}\int_{\mathbb{R}^{N}}Q_{0}^{\frac{1}{p}}v\mathbf{R}^{s}(Q_{0}^{\frac{1}{p}}v)\mathrm{d}x,~~v\in L^{p'}(\mathbb{R}^{N})
\end{equation}
and the corresponding Nehari manifold
\begin{equation}
\mathcal{N}_{0}:=\{v\in L^{p'}(\mathbb{R}^{N})\setminus\{0\}:J'_{0}(v)v=0\},
\end{equation}
associated to the limit problem
\begin{equation}\label{limit5}
(-\Delta)^{s} u-u=Q_{0}|u|^{p-2}u,~~x\in\mathbb{R}^{N}.
\end{equation}
Lemma \ref{mountain3} implies that the level $c_{0}:=\mathop{\mathrm{inf}}\limits_{\mathcal{N}_{0}}J_{0}$ is attained and coincides with the least-energy level, i.e.,
\begin{equation}
c_{0}=\mathrm{inf}\{J_{0}(v):v\in L^{p'}(\mathbb{R}^{N}),v\neq0~~\mathrm{and}~J'_{0}(v)=0\}.
\end{equation}

Denote the set of maximum points of $Q$ by
\begin{equation}
M:=\{x\in\mathbb{R}^{N}:Q(x)=Q_{0}\}.
\end{equation}
It then follows from (Q0) and (Q1) that $M\neq \emptyset$. We start by studying the projection on the Nehari manifold of truncation of of translated and rescaled ground states of $J_{0}$. Take a cutoff function $\eta\in C_{c}^{\infty}(\mathbb{R}^{N})$, $0\leq\eta\leq1$, such that $\eta\equiv1$ in $B_{1}(0)$ and $\eta\equiv0$ in $\mathbb{R}^{N}\setminus B_{2}(0)$. For $y\in M$, $\varepsilon>0$ we let
\begin{equation}\label{cutoff}
\varphi_{\varepsilon,y}(x):=\eta(\varepsilon x-y)w(x-\varepsilon^{-1}y),
\end{equation}
where $w\in L^{p'}(\mathbb{R}^{N})$ is some fixed least-energy critical point of $J_{0}$.

\begin{lem}\label{limit3}
There is $\varepsilon^{\ast}>0$ such that for all $0<\varepsilon\leq\varepsilon^{\ast}$, $y\in M$, a unique $t_{\varepsilon,y}>0$ satisfying $t_{\varepsilon,y}\varphi_{\varepsilon,y}\in\mathcal{N}_{\varepsilon}$ exists. Moreover,
\begin{equation}
\mathop{\mathrm{lim}}\limits_{\varepsilon\longrightarrow 0^{+}}J_{\varepsilon}(t_{\varepsilon,y}\varphi_{\varepsilon,y})=c_{0},~~~\mathrm{uniformly~for~}y\in M.
\end{equation}
\end{lem}
\begin{proof}
Since $M$ is compact and $Q$ is continuous by assumption, we have $Q(y+\varepsilon\cdot)\eta(\varepsilon\cdot)w\longrightarrow Q_{0}w$ in $L^{p'}(\mathbb{R}^{N})$ as $\varepsilon\longrightarrow0^{+}$, uniformly with respect to $y\in M$,  Consequently, as $\varepsilon\longrightarrow0^{+}$,
\begin{equation}
\begin{aligned}
\int_{\mathbb{R}^{N}}Q_{\varepsilon}^{\frac{1}{p}}\varphi_{\varepsilon,y}\mathbf{R}^{s}(Q_{\varepsilon}^{\frac{1}{p}}\varphi_{\varepsilon,y})\mathrm{d}x
&=\int_{\mathbb{R}^{N}}Q^{\frac{1}{p}}(y+\varepsilon z)\eta(\varepsilon z)w(z)\mathbf{R}^{s}(Q^{\frac{1}{p}}(y+\varepsilon\cdot)\eta(\varepsilon\cdot)w)(x)\mathrm{d}z\\
&\longrightarrow \int_{\mathbb{R}^{N}}Q_{0}^{\frac{1}{p}}w\mathbf{R}^{s}(Q_{0}^{\frac{1}{p}}w)\mathrm{d}z=(\frac{1}{p'}-\frac{1}{2})^{-1}c_{0}>0,
\end{aligned}
\end{equation}   
uniformly for $y\in M$. Therefore, for all $y\in M$ and $\varepsilon>0$ small enough, we deduce that $\varphi_{\varepsilon,y}\in U_{\varepsilon}^{+}$, this implies the first assertion with $t_{\varepsilon,y}$ given by (\ref{t}). In addition, for all $y\in M$,
\begin{equation}
\int_{\mathbb{R}^{N}}|\varphi_{\varepsilon,y}|^{p'}\mathrm{d}x=\int_{\mathbb{R}^{N}}|\eta(\varepsilon z)w(z)|^{p'}\mathrm{d}z\longrightarrow \int_{\mathbb{R}^{N}}|w|^{p'}\mathrm{d}z=(\frac{1}{p'}-\frac{1}{2})^{-1}c_{0},~~\mathrm{as}~\varepsilon\longrightarrow0^{+}.
\end{equation}
As a consequence, $t_{\varepsilon,y}\longrightarrow 1$ as $\varepsilon\longrightarrow0^{+}$, uniformly for $y\in M$, and we obtain $J_{\varepsilon}(t_{\varepsilon,y}\varphi_{\varepsilon,y})\longrightarrow c_{0}$ as $\varepsilon\longrightarrow 0^{+}$, uniformly for $y\in M$. The second assertion follows.
\end{proof}

\begin{lem}\label{limit4}
For all $\varepsilon>0$ there holds $c_{0}\leq c_{\varepsilon}$. Moreover, $\mathop{\mathrm{lim}}\limits_{\varepsilon\longrightarrow0^{+}}c_{\varepsilon}=c_{0}$.
\end{lem}

\begin{proof}
Consider $v_{\varepsilon}\in\mathcal{N}_{\varepsilon}$ and set $v_{0}:=(\frac{Q_{\varepsilon}}{Q_{0}})^{\frac{1}{p}}v_{\varepsilon}$. Notice that $|v_{0}|\leq |v_{\varepsilon}|$ a.e. on $\mathbb{R}^{N}$. Since $v_{\varepsilon}\in U_{\varepsilon}^{+}$, we find
\begin{equation}
\int_{\mathbb{R}^{N}}Q_{0}^{\frac{1}{p}}v_{0}\mathbf{R}^{s}(Q_{0}^{\frac{1}{p}}v_{0})=
\int_{\mathbb{R}^{N}}Q_{\varepsilon}^{\frac{1}{p}}v_{\varepsilon}\mathbf{R}^{s}(Q_{\varepsilon}^{\frac{1}{p}}v_{\varepsilon})
>0,
\end{equation}
i.e., $v_{0}\in U_{0}^{+}$. By (\ref{t}) we deduce
\begin{equation}
t_{\varepsilon}^{2-p'}=\frac{\int_{\mathbb{R}^{N}}|v_{0}|^{p'}\mathrm{d}x}{\int_{\mathbb{R}^{N}}Q_{0}^{\frac{1}{p}}v_{0}\mathbf{R}^{s}(Q_{0}^{\frac{1}{p}}v_{0})\mathrm{d}x}
\leq \frac{\int_{\mathbb{R}^{N}}|v_{\varepsilon}|^{p'}\mathrm{d}x}{\int_{\mathbb{R}^{N}}Q_{0}^{\frac{1}{p}}v_{\varepsilon}\mathbf{R}^{s}(Q_{\varepsilon}^{\frac{1}{p}}v_{\varepsilon})\mathrm{d}x}=1.
\end{equation}
This implies that $t_{\varepsilon}v_{0}\in\mathcal{N}_{0}$. Follow the definition of the dual functional, we yield that
\begin{equation}
c_{0}\leq J_{0}(t_{\varepsilon}v_{0})=(\frac{1}{p'}-\frac{1}{2})t_{\varepsilon}^{p'}\int_{\mathbb{R}^{N}}|v_{0}|^{p'}\mathrm{d}x
\leq(\frac{1}{p'}-\frac{1}{2})\int_{\mathbb{R}^{N}}|v_{\varepsilon}|^{p'}\mathrm{d}x=J_{\varepsilon}(v_{\varepsilon}).
\end{equation}
Since $v_{\varepsilon}\in\mathcal{N}_{\varepsilon}$ was arbitrarily chosen, we conclude that $c_{0}\leq \mathop{\mathrm{inf}}\limits_{\mathcal{N}_{\varepsilon}}=c_{\varepsilon}$. On the other hand, Lemma \ref{limit3} gives for $y\in M$, $c_{\varepsilon}\leq J_{\varepsilon}(t_{\varepsilon,y}\varphi_{\varepsilon,y})\longrightarrow c_{0}$ as $\varepsilon\longrightarrow 0^{+}$. Hence, $\mathop{\mathrm{lim}}\limits_{\varepsilon\longrightarrow0^{+}}c_{\varepsilon}=c_{0}$, as claimed.
\end{proof}

Now, consider the energy functional $J_{\infty}:L^{p'}(\mathbb{R}^{N})\longrightarrow\mathbb{R}$ given by
\begin{equation}
J_{\infty}(v)=\frac{1}{p'}\int_{\mathbb{R}^{N}}|v|^{p'}\mathrm{d}x-\frac{1}{2}\int_{\mathbb{R}^{N}}Q_{\infty}^{\frac{1}{p}}v\mathbf{R}^{s}\big(Q_{\infty}^{\frac{1}{p}}v\big)\mathrm{d}x,~~~v\in L^{p'}(\mathbb{R}^{N}).
\end{equation}
The corresponding Nehari manifold
\begin{equation}
\mathcal{N}_{\infty}:=\{v\in L^{p'}(\mathbb{R}^{N})\setminus\{0\}:J'_{\infty}(v)v=0\},
\end{equation}
has the same structure as $\mathcal{N}_{\varepsilon}$ and, since $Q_{\infty}$ is constant, Lemma
\ref{mountain3} implies that $c_{\infty}:=\mathop{\mathrm{inf}}\limits_{\mathcal{N}_{\infty}}J_{\infty}$ is attained and coincides with the least energy level for nontrivial critical points of $J_{\infty}$.

\begin{Prop}\label{level}
There is $\varepsilon_{0}>0$ such that for all $\varepsilon<\varepsilon_{0}$, $c_{\varepsilon}<c_{\infty}$.
\end{Prop}
\begin{proof}
By Lemma \ref{limit4} and Condition (Q1), there is $\varepsilon_{0}>0$ such that $c_{\varepsilon}<c_{\infty}$ for all $0<\varepsilon<\varepsilon_{0}$.
\end{proof}

\section{Exsitence of dual ground ststes}

In this section, we proof the $(PS)_{c_{\varepsilon}}$ condition for the energy functional $J_{\varepsilon}$. Since the resolvent Helmholtz operator is not positive definite, we need to handle the nonlocal interaction between functions with disjoint support.
\begin{lem}\label{nonlocal}
there exists a constant $C=C(N,p)>0$ such that for any $R>0$, $r\geq1$ and $u,v\in L^{p'}(\mathbb{R}^{N})$ with $\mathrm{supp}(u)\subset B_{R}$ and $\mathrm{supp}(v)\subset \mathbb{R}^{N}\setminus B_{R+r}$,
\begin{equation}
\big|\int_{\mathbb{R}^{N}}u\mathbf{R}^{s}v~\mathrm{d}x\big|\leq Cr^{-\lambda_{p}}||u||_{p'}||v||_{p'},~~~\mathrm{where}~\lambda_{p}=\frac{N-1}{2}-\frac{N+1}{p}.
\end{equation}
\end{lem}

\begin{proof}
Let $\mathcal{R}^{s}$ denote the resolvent of the Fractional Helmholtz eqaution, which is given by the convolution with the kernel $K(x)$ (see \cite{Yao2016} and \cite{Shen2023} for more details).
Since $\mathbf{R}^{s}$ is the real part of $\mathcal{R}^{s}$ and since $u,v$ are real-valued, we prove the lemma for the nonlocal term $\int_{\mathbb{R}^{N}}v\mathcal{R}^{s}u\mathrm{d}x$. By density, it suffices to prove the estimate for Schwartz function.

Let $M_{R+r}:=\mathbb{R}^{N}\setminus B_{R+r}$ and let $u,v\in\mathcal{S}(\mathbb{R}^{N})$ be such that $\mathrm{supp}(u)\subset B_{R}$ and $\mathrm{supp}(v)\subset M_{R+r}$. The symmetry of the operator $\mathcal{R}^{s}$ and H\"{o}lder's inequality gives
\begin{equation}\label{Holder}
\big|\int_{\mathbb{R}^{N}}u\mathcal{R}^{s}v\mathrm{d}x\big|=\big|\int_{\mathbb{R}^{N}}v\mathcal{R}^{s}u\mathrm{d}x\big|\leq||v||_{p'}||K\ast u||_{L^{p}(M_{R+r})}.
\end{equation}
This reduces us to estimating the second factor on the right-hand side. For this, we decomposition $K$ as follow. Fix $\psi\in\mathcal{S}(\mathbb{R}^{N})$ such that $\widehat{\psi}\in\mathcal{C}_{c}^{\infty}(\mathbb{R}^{N})$ is radial, $0\leq\widehat{\psi}\leq1$, $\widehat{\psi}(\xi)=1$ for $||\xi|-1|\leq\frac{1}{6}$ and $\widehat{\psi}(\xi)=0$ for $||\xi|-1|\geq \frac{1}{4}$. Writing $K=K_{1}+K_{2}$ with
\begin{equation}
K_{1}:=(2\pi)^{-\frac{N}{2}}(
\psi\ast K),~~~~K_{2}:=K-K_{1}.
\end{equation}
It follows from the estimate in \cite{Shen2023} that
\begin{equation}\label{Phi1}
|K_{1}(x)|\leq C_{0}(1+|x|)^{\frac{1-N}{2}}~~\mathrm{for}~x\in\mathbb{R}^{N},
\end{equation}
\begin{equation}\label{Phi2}
\mathrm{and}~~|K_{2}(x)|\leq C_{0}|x|^{-N}~~\mathrm{for}~x\neq0.
\end{equation}
Since the support of $u$ is contained in $B_{R}$, we find
\begin{equation}
\begin{aligned}
||K_{2}\ast u||_{L^{p}(M_{R+r})}&\leq\big[\int_{|x|\geq R+r}\big(\int_{|y|\leq R}|K_{2}(x-y)||u(y)|\mathrm{d}y\big)^{p}\mathrm{d}x\big]^{\frac{1}{p}}\\
&\leq\big[\int_{\mathbb{R}^{N}}\big(\int_{|x-y|\geq r}|K_{2}(x-y)||u(y)|\mathrm{d}y\big)^{p}\mathrm{d}x\big]^{\frac{1}{p}}\\
&=||(1_{M_{r}}|K_{2}|\ast|u|)||_{p}\leq ||1_{M_{r}}K_{2}||_{\frac{p}{2}}||u||_{p'}.
\end{aligned}
\end{equation}
Moreover, (\ref{Phi2}) gives
\begin{equation}
||1_{M_{r}}K_{2}||_{\frac{p}{2}}\leq C_{0}\big(\omega_{N}\int_{r}^{\infty}s^{N-1-\frac{Np}{2}}ds\big)^{\frac{2}{p}}\leq Cr^{-\frac{N(p-2)}{p}}\leq Cr^{-\lambda_{p}},
\end{equation}
since $r\geq 1$, and therefore
\begin{equation}\label{Phi22}
||K_{2}\ast u||_{L^{p}(M_{R+r})}\leq Cr^{-\lambda_{p}}||u||_{p'}.
\end{equation}
It remains to prove the estimate for $K_{1}$. Fix a radial function $K\in\mathcal{S}(\mathbb{R}^{N})$ such that $\widehat{K}\in\mathcal{C}_{c}^{\infty}(\mathbb{R}^{N})$ is radial,
$0\leq\widehat{K}\leq 1$, $\widehat{K}(\xi)=1$ for $||\xi|-1|\leq\frac{1}{4}$ and $\widehat{K}(\xi)=0$ for $||\xi|-1|\geq\frac{1}{2}$, let $\tilde{u}:=K\ast u\in\mathcal{S}(\mathbb{R}^{N})$, we then have $K_{1}\ast u=(2\pi)^{-\frac{N}{2}}(K_{1}\ast\tilde{u})$, since $\widehat{K_{1}}\widehat{K}=\widehat{K_{1}}$ by construction. Now write
\begin{equation}
K_{1}\ast\tilde{u}=\big[1_{B_{\frac{r}{2}}}K_{1}\big]\ast\tilde{u}+\big[1_{M_{\frac{r}{2}}}K_{1}\big]\ast\tilde{u}
\end{equation}
and let $g_{r}:=\big[1_{B_{\frac{r}{2}}}K_{1}\big]\ast K$. Since $\mathrm{supp}(u)\subset B_{R}$, we find as above
\begin{equation}
\begin{aligned}
||1_{M_{r}}g_{r}||_{\frac{p}{2}}^{\frac{p}{2}}&\leq C_{0}^{\frac{p}{2}}\int_{|x|\geq r}\big(\int_{|y|\leq \frac{r}{2}}|K(x-y)|\mathrm{d}y\big)^{\frac{p}{2}}\mathrm{d}x\\
&\leq C\int_{|x|\geq r}\big(\int_{|y|\leq\frac{r}{2}}|x-y|^{-m}\mathrm{d}y\big)^{\frac{p}{2}}\mathrm{d}x\leq C|B_{\frac{r}{2}}|^{\frac{p}{2}}\int_{|x|\geq r}(|x|-\frac{r}{2})^{-\frac{mp}{2}}dx\\
&=Cr^{\frac{(N-m)p}{2}+N}\int_{|z|\geq1}(|z|-\frac{1}{2})^{-\frac{mp}{2}}\mathrm{d}z=Cr^{\frac{(N-m)p}{2}+N},
\end{aligned}
\end{equation}
where $C$ is independent of $r$ and where $m$ may be fixed so large that $\frac{(m-N)p}{2}-N\geq\lambda_{p}$. As a consequence of \cite[Proposition 3.3]{Evequoz2015}, we have moreover
\begin{equation}
||\big[1_{M_{\frac{r}{2}}}K_{1}\big]\ast \tilde{u}||_{L^{p}(M_{R+r})}\leq||\big[1_{M_{\frac{r}{2}}}K_{1}\ast\tilde{u}\big]||_{p}\leq Cr^{-\lambda_{p}}||\tilde{u}||_{p'}\leq Cr^{-\lambda_{p}}||u||_{p'}
\end{equation}
and we conclude that
\begin{equation}\label{Phi11}
||K_{1}\ast u||_{L^{p}(M_{R+r})}\leq Cr^{-\lambda_{p}}||u||_{p'}.
\end{equation}
Combining (\ref{Holder}), (\ref{Phi11}) and (\ref{Phi22}) yields the claim.
\end{proof}
\begin{lem}\label{ps condition}
Let $\varepsilon>0$ and assume $Q_{\infty}>0$ and $c_{\varepsilon}<c_{\infty}$. Then $J_{\varepsilon}$ satisfies the Palais-Smale condition on $\mathcal{N}_{\varepsilon}$ at level below $c_{\infty}$, i.e., every sequence $(v_{n})_{n}\subset\mathcal{N}_{\varepsilon}$ such that $J_{\varepsilon}(v_{n})\longrightarrow d<c_{\infty}$ and $(J_{\varepsilon}|{\mathcal{N}_{\varepsilon}})'(v_{n})\longrightarrow 0$ as $n\longrightarrow\infty$ has a convergent subsequence.
\end{lem}

\begin{proof}
Since $c_{\varepsilon}<c_{\infty}$, the set $\{v\in\mathcal{N}_{\varepsilon}:J_{\varepsilon}(v)<c_{\infty}\}$ is not empty. If $d<c_{\varepsilon}$, all is done. It remains to consider the case $c_{\varepsilon}\leq d<c_{\infty}$. Let $(v_{n})_{n}$ be a $(\mathrm{PS})_{d}$-sequence  for $J_{\varepsilon}|{\mathcal{N}_{\varepsilon}}$. Since $\mathcal{N}_{\varepsilon}$ is a natural constraint and a $C^{1}-$manifold, we find that $(v_{n})_{n}$ is a $(PS)_{d}$-sequence for the unconstrained functional $J_{\varepsilon}$. Using Lemma \ref{PS}, we obtain that (up to a subsequence) $v_{n}\rightharpoonup v$ and $1_{B_{R}}v_{n}\longrightarrow 1_{B_{R}}v$ in $L^{p'}(\mathbb{R}^{N})$ for all $R>0$, where $v\in L^{p'}(\mathbb{R}^{N})$ is a critical point of $J_{\varepsilon}$ with $J_{\varepsilon}(v)\leq d$. In order to conclude that $v_{n}\longrightarrow v$ strongly in $L^{p'}(\mathbb{R}^{N})$, it suffices to show that
\begin{equation}\label{suffice}
\forall~\zeta>0,~~\exists~R>0~~~\mathrm{such~that}~~\int_{|x|>R}|v_{n}|^{p'}dx<\zeta,~~~\forall~~n.
\end{equation}
We prove (\ref{suffice}) by contradiction. Assuming that there exists a subsequence $(n_{n_{k}})_{k}$ and $\zeta_{0}>0$ such that
\begin{equation}\label{contradiction}
\int_{|x|>k}|v_{n_{k}}|^{p'}\geq \zeta_{0},~~~\forall~k.
\end{equation}

Firstly, for a annular region, we claim that
\begin{equation}\label{claim}
\forall~\eta>0~~\mathrm{and}~~\forall~R>0,~~~\exists~r>R~~~\mathrm{such~that}~~~\mathop{\mathrm{lim~inf}}\limits_{n\longrightarrow\infty}
\int_{r<|x|<2r}|v_{n}|^{p'}\mathrm{d}x<\eta.
\end{equation}
Otherwise, for every $m>R_{0}$, $n_{0}=n_{0}(m)$, we can find $\eta_{0},R_{0}$ such that  $\int_{m<|x|<2m}|v_{n}|^{p'}\mathrm{d}x\geq \eta_{0}$ for all $n\geq n_{0}$. Without loss of generality, we assume that $n_{0}(m+1)\geq n_{0}(m)$ for all $m$. Hence, for every $l\in\mathbb{N}$ there is $N_{0}=N_{0}(l)$ such that
\begin{equation}
\int_{\mathbb{R}^{N}}|v_{n}|^{p'}\mathrm{d}x\geq \mathop{\sum}\limits_{k=0}^{l-1}\int_{2^{k}([R_{0}]+1)<|x|<2^{k+1}([R_{0}]+1)}|v_{n}|^{p'}\mathrm{d}x\geq l\eta_{0},~~~\forall~n\geq N_{0}.
\end{equation}
Letting $l\longrightarrow\infty$, we obtain a contradiction to the fact that $(v_{n})_{n}$ is bounded and this gives (\ref{claim}).

Now fix $0<\eta<\mathrm{min}\{1,(\frac{\zeta_{0}}{3C_{1}})^{p'}\}$, where $C_{1}=2C(N,p)||Q||_{\infty}^{\frac{2}{p}}\mathrm{max}\{1,\mathop{\mathrm{sup}}\limits_{k\in\mathbb{N}}||v_{n_{k}}||_{p'}^{2}\}$,
the constant $C(N,p)$ being chosen such that Lemma \ref{nonlocal} holds and $||\mathbf{R}^{s}v||_{p}\leq C(N,p)||v||_{p'}$ for all $u\in L^{p'}(\mathbb{R}^{N})$. By definition of $Q_{\infty}$ and since $\varepsilon>0$ is fixed, there exists $R(\eta)>0$ such that
\begin{equation}
Q_{\varepsilon}\leq Q_{\infty}+\eta~~~\mathrm{for~all}~|x|\geq R(\eta).
\end{equation}
Also, from (\ref{claim}), we can find $r>\mathrm{max}\{R(\eta),\eta^{-\frac{1}{\lambda_{p}}}\}$ and a subsequence, still denoted by $(v_{n_{k}})_{k}$, such that
\begin{equation}
\int_{r<|x|<2r}|v_{n_{k}}|^{p'}\mathrm{d}x<\eta~~~~\mathrm{for~all}~k.
\end{equation}
Setting $w_{n_{k}}:=1_{\{|x|\geq2r\}}v_{n_{k}}$ we can write for all $k$
\begin{equation}  
\begin{aligned}
\big|J'_{\varepsilon}(v_{n_{k}})w_{n_{k}}-J'_{\varepsilon}(w_{n_{k}})w_{n_{k}}\big|
&=\big|\int_{|x|<r}Q_{\varepsilon}^{\frac{1}{p}}
v_{n_{k}}\mathbf{R}^{s}\big(Q_{\varepsilon}^{\frac{1}{p}}w_{n_{k}}\big)\mathrm{d}x+\int_{r<|x|<2r}Q_{\varepsilon}^{\frac{1}{p}}v_{n_{k}}\mathbf{R}^{s}\big(Q_{\varepsilon}^{\frac{1}{p}}w_{n_{k}}\big)\mathrm{d}x\big|\\
&\leq C(N,p)r^{-\lambda_{p}}||Q||_{\infty}^{\frac{2}{p}}||v_{n_{k}}||_{p'}^{2}+C(N,p)||Q||_{\infty}^{\frac{2}{p}}||v_{n_{k}}||_{p'}\big(\int_{r<|x|<2r}|v_{n_{k}}|^{p'}\mathrm{d}x\big)^{\frac{1}{p'}}\\
&\leq C_{1}\eta^{\frac{1}{p'}},
\end{aligned}
\end{equation}
using Lemma \ref{nonlocal}. In addition, by (\ref{contradiction}) and the definition of $w_{n_{k}}$, there holds
\begin{equation}
\int_{\mathbb{R}^{N}}|w_{n_{k}}|^{p'}\mathrm{d}x\geq \zeta_{0}~~~\mathrm{for~all}~k\geq 2r.
\end{equation}

Recalling our choice of $\eta$, we know that $C_{1}\eta^{\frac{1}{p'}}<\frac{\zeta_{0}}{3}$,
and we find some $k_{0}=k_{0}(r,\eta,\zeta_{0})\geq 2r$ such that
\begin{equation}\label{part1}
\begin{aligned}
\int_{\mathbb{R}^{N}}Q_{\varepsilon}^{\frac{1}{p}}w_{n_{k}}\mathbf{R}^{s}\big(Q_{\varepsilon}^{\frac{1}{p}}w_{n_{k}}\big)\mathrm{d}x
&=\int_{\mathbb{R}^{N}}|w_{n_{k}}|^{p'}\mathrm{d}x-J'_{\varepsilon}(v_{n_{k}})w_{n_{k}}+[J'_{\varepsilon}(v_{n_{k}})w_{n_{k}}-J'_{\varepsilon}(w_{n_{k}})w_{n_{k}}]\\
&\geq\int_{\mathbb{R}^{N}}|w_{n_{k}}|^{p'}\mathrm{d}x-|J'_{\varepsilon}(v_{n_{k}})w_{n_{k}}|-C_{1}\eta^{\frac{1}{p'}}\geq\frac{\zeta_{0}}{2},~~~\mathrm{for~~all}~k\geq k_{0},
\end{aligned}
\end{equation}
since $J'_{\varepsilon}(v_{n_{k}})w_{n_{k}}\longrightarrow0$ as $k\longrightarrow\infty$.

For $k\geq k_{0}$, let now $\tilde{w}_{k}:=(\frac{Q_{\varepsilon}}{Q_{\infty}})^{\frac{1}{p}}w_{n_{k}}$ and notice that $|\tilde{w}_{k}|\leq (1+\frac{\eta}{Q_{\infty}})^{\frac{1}{p}}|w_{n_{k}}|$. In view of (\ref{part1}), there is $t_{k}^{\infty}>0$, for which $t^{\infty}_{k}\tilde{w}_{k}\in\mathcal{N}_{\infty}$ and there holds
\begin{equation}
\begin{aligned}
(t_{k}^{\infty})^{2-p'}&\leq \frac{(1+\frac{\eta}{Q_{\infty}})^{p'-1}\int_{\mathbb{R}^{N}}|w_{n_{k}}|^{p'}\mathrm{d}x}{\int_{\mathbb{R}^{N}}Q_{\varepsilon}^{\frac{1}{p}}w_{n_{k}}\mathbf{R}^{s}(Q_{\varepsilon}^{\frac{1}{p}}w_{n_{k}})\mathrm{d}x}\\
&\leq(1+\frac{\eta}{Q_{\infty}})^{p'-1}\big(1+\frac{|J'_{\varepsilon}(w_{n_{k}})w_{n_{k}}|+C_{1}\eta^{\frac{1}{p'}}}{\int_{\mathbb{R}^{N}}Q_{\varepsilon}^{\frac{1}{p}}w_{n_{k}}\mathbf{R}^{s}(Q_{\varepsilon}^{\frac{1}{p}}w_{n_{k}})\mathrm{d}x}   \big)\\
&\leq(1+\frac{\eta}{Q_{\infty}})^{p'-1}\big(1+\frac{2|J'_{\varepsilon}(w_{n_{k}})w_{n_{k}}+2C_{1}\eta^{\frac{1}{p'}}|}{\zeta_{0}}\big).
\end{aligned}
\end{equation}

Since $v_{n_{k}}\in\mathcal{N}_{\varepsilon}$, there holds
\begin{equation}\label{part2}
\int_{\mathbb{R}^{N}}|w_{n_{k}}|^{p'}\mathrm{d}x\leq \int_{\mathbb{R}^{N}}|v_{n_{k}}|^{p'}\mathrm{d}x=(\frac{1}{p'}-\frac{1}{2})^{-1}J_{\varepsilon}(v_{n_{k}}).
\end{equation}
Consequently, for all $k\geq k_{0}$,
\begin{equation}
\begin{aligned}
c_{\infty}&\leq J_{\infty}(t_{k}^{\infty}\tilde{w}_{k})\\
&\leq(\frac{1}{p'}-\frac{1}{2})(t_{k}^{\infty})^{p'}(1+\frac{\eta}{Q_{\infty}})^{p'-1}\int_{\mathbb{R}^{N}}|w_{n_{k}}|^{p'}\mathrm{d}x\\
&\leq(1+\frac{\eta}{Q_{\infty}})^{\frac{2(p'-1)}{2-p'}}\big(1+\frac{2|J'_{\varepsilon}(w_{n_{k}})w_{n_{k}}+2C_{1}\eta^{\frac{1}{p'}}|}{\zeta_{0}}\big)^{\frac{p'}{2-p'}}J_{\varepsilon}(v_{n_{k}}).
\end{aligned}
\end{equation}
Letting $k\longrightarrow\infty$, we find
\begin{equation}
c_{\infty}\leq(1+\frac{\eta}{Q_{\infty}})^{\frac{2(p'-1)}{2-p'}}\big(1+\frac{2C_{1}\eta^{\frac{1}{p'}}}{\zeta_{0}}\big)^{\frac{p'}{2-p'}}d,
\end{equation}
and letting $\eta\longrightarrow0$ we obtain
\begin{equation}
c_{\infty}\leq d,
\end{equation}
which contradicts the assumption $d<\infty$ and prove (\ref{suffice}). From this, we conclude the strong convergence $v_{n}\longrightarrow v$ in $L^{p'}(\mathbb{R}^{N})$ and the assertion follows.
\end{proof}

\begin{proof}[\bf Proof of Theorem 1.1 (i)]
Fix $\varepsilon_{0}$ in Proposition (\ref{level}). For any $\varepsilon\leq\varepsilon_{0}$, using the fact that $\mathcal{N}_{\varepsilon}$ is a $C^{1}-$submanifold of $L^{p'}(\mathbb{R}^{N})$, we obtain from Ekeland's variational principle the existence of Palais-Smale sequence for $J_{\varepsilon}$ on $\mathcal{N}_{\varepsilon}$, at level $c_{\varepsilon}$, and by Lemma \ref{ps condition}, $c_{\varepsilon}$ is attained.
\end{proof}

\section{Concentration of dual ground states}

To show the concentration behaviour of the solutions of (\ref{main1}), we first prove a representation lemma for the Palais-Smale sequences of the functional $J_{\varepsilon}$, which is in the spirit of and Benci and Cerami \cite{Benci1987}. A crucial ingredient related to the nonlocal quadratic part of the energy functional is the nonvanishing theory proved in \cite[Sect.4]{Shen2023}. For simplicity, we drop the subscript $\varepsilon$.

\begin{lem}\label{representation}
Suppose $Q\equiv Q(0)>0$ on $\mathbb{R}^{N}$. Consider for some $d>0$ a $(PS)_{d}$-sequence $(v_{n})_{n}\subset L^{p'}(\mathbb{R}^{N})$ for $J$. Then there is an integer $m\geq1$, critical points $w^{(1)},...,w^{(m)}$ of $J$ and sequence $(x_{n}^{(1)})_{n},....,(x_{n}^{(m)})_{n}\subset\mathbb{R}^{N}$ such that (up to a subsequence)
\begin{equation}
\left \{
\begin{aligned}
&||v_{n}-\mathop{\sum}\limits_{j=1}^{m}w^{j}(\cdot-x_{n}^{j})||_{p'}\longrightarrow 0,~~~\mathrm{as}~n\longrightarrow\infty,\\
&|x_{n}^{(i)}-x_{n}^{(j)}|\longrightarrow\infty~~~\mathrm{as}~n\longrightarrow\infty,~\mathrm{if}~i\neq j,\\
&\mathop{\sum}\limits_{j=1}^{m}J(w^{(j)})=d.\\
\end{aligned}
\right.
\end{equation}
\end{lem}
\begin{proof}
For any bounded $(\mathrm{PS})_{d}$-sequenc $(v_{n})_{n}$, we have
\begin{equation}
\mathop{\mathrm{lim}}\limits_{n\longrightarrow\infty}\int_{\mathbb{R}^{N}}Q^{\frac{1}{p}}v_{n}\mathbf{R}^{s}\big(Q^{\frac{1}{p}}v_{n}\big)\mathrm{d}x
=\frac{2p'}{2-p'}\mathop{\mathrm{lim}}\limits_{n\longrightarrow\infty}\big[J(v_{n})-\frac{1}{p'}J'(v_{n})v_{n}\big]=\frac{2p'd}{2-p'}>0.
\end{equation}
It then follows from the nonvanishing theorem \cite[Theorem 4.1]{Shen2023} that there are $R,\zeta>0$ and a sequence $(x_{n}^{(1)})_{n}$ such that, up to a subsequence,
\begin{equation}
\int_{B_{r}(x_{n}^{(1)})}|v_{n}|^{p'}\mathrm{d}x\geq\zeta>0~~~\mathrm{for~all}~n.
\end{equation}
Setting $v_{n}^{(1)}=v_{n}(\cdot+x_{n}^{(1)})$, then by the invariance of the energy functional, $(v_{n}^{1})_{n}$ is also a $(\mathrm{PS})_{d}$-sequence for $J$. By Lemma \ref{PS}, going to a further subsequence, we may assume $v_{n}^{(1)}\rightharpoonup w^{(1)}$
 weakly, $1_{B_{R}}v_{n}^{(1)}\longrightarrow 1_{B_{R}}w^{(1)}$ strongly in $L^{p'}(\mathbb{R}^{N})$, and $J(w^{(1)})\leq \mathop{\mathrm{lim}}\limits_{n\longrightarrow\infty}J(v_{n}^{(1)})=d$. These properties and the definition of $v_{n}^{(1)}$ imply that $w^{(1)}$ is a nontrivial critical point of $J$.

 If $J(w^{(1)})=d$, we obtain
 \begin{equation}
 \begin{aligned}
 \big(\frac{1}{p'}-\frac{1}{2}\big)||w^{(1)}||_{p'}^{p'}&=J(w^{(1)})-\frac{1}{2}J'(w^{(1)})w^{(1)}\\
 &=d=\mathop{\mathrm{lim}}\limits_{n\longrightarrow\infty}\big[J(v_{n})-\frac{1}{2}J'(v_{n})v_{n}\big]=\big(\frac{1}{p'}-\frac{1}{2}\big)
 \mathop{\mathrm{lim}}\limits_{n\longrightarrow\infty}||v_{n}||_{p'}^{p'},
 \end{aligned}
 \end{equation}
i.e., $v_{n}^{(1)}\longrightarrow w^{(1)}$ strongly in $L^{p'}(\mathbb{R}^{N})$, then the lemma is proved.

Otherwise, $J(w^{(1)})<d$ and we set $v_{n}^{(2)}=v_{n}^{(1)}-w^{(1)}$. The weak convergence $v_{n}^{(1)}\rightharpoonup w^{(1)}$ then implies
\begin{equation}
\begin{aligned}
\int_{\mathbb{R}^{N}}Q^{\frac{1}{p}}v_{n}^{(2)}\mathbf{R}^{s}\big(Q^{\frac{1}{p}}v_{n}^{(2)}\big)\mathrm{d}x&=\int_{\mathbb{R}^{N}}Q^{\frac{1}{p}}v_{n}^{(1)}\mathbf{R}^{s}\big(Q^{\frac{1}{p}}v_{n}^{(1)}\big)\mathrm{d}x\\
&-\int_{\mathbb{R}^{N}}Q^{\frac{1}{p}}w_{n}^{(1)}\mathbf{R}^{s}\big(Q^{\frac{1}{p}}w_{n}^{(1)}\big)\mathrm{d}x+o(1),
\end{aligned}
\end{equation}
as $n\longrightarrow\infty$. Moreover, by the Br\'{e}zis-Lieb Lemma \cite{Brezis1983},
\begin{equation}
\int_{\mathbb{R}^{N}}|v_{n}^{(2)}|^{p'}\mathrm{d}x=\int_{\mathbb{R}^{N}}|v_{n}^{(1)}|^{p'}\mathrm{d}x-\int_{\mathbb{R}^{N}}|w^{(1)}|^{p'}\mathrm{d}x+o(1),~~\mathrm{as}~n\longrightarrow\infty.
\end{equation}
These properties and the translation invariance of $J$ together give
\begin{equation}
J(v_{n}^{(2)})=J(v_{n}^{(1)})-J(w^{(1)})+o(1)=d-J(w^{(1)})+o(1),~~\mathrm{as}~n\longrightarrow\infty.
\end{equation}
Since by Lemma \ref{PS}, $1_{B_{r}}v_{n}^{(1)}\longrightarrow1_{B_{r}}w^{(1)}$ strongly in $L^{p'}(\mathbb{R}^{N})$ for all $r>0$, we find
\begin{equation}
1_{B_{r}}|v_{n}^{(2)}|^{p'-2}v_{n}^{(2)}-1_{B_{r}}|v_{n}^{(1)}|^{p'-2}v_{n}^{(1)}+1_{B_{r}}|w^{(1)}|^{p'-2}w^{(1)}\longrightarrow 0~~\mathrm{in}~L^{p}(\mathbb{R}^{N}),~~\mathrm{as}~n\longrightarrow\infty.
\end{equation}
On the other hand, since $||a|^{q-1}a-|b|^{q-1}b|\leq 2^{1-q}|a-b|^{q}$ for all $a,b\in\mathbb{R}$ and $0<q<1$, it follows that
\begin{equation}
\int_{\mathbb{R}^{N}\setminus B_{r}}\big||v_{n}^{(2)}|^{p'-2}v_{n}^{(2)}-|v_{n}^{(1)}|^{p'-2}v_{n}^{(1)}\big|^{p}\mathrm{d}x\leq 2^{(2-p')p}\int_{\mathbb{R}^{N}\setminus B_{r}}|w^{(1)}|^{p'}\mathrm{d}x\longrightarrow 0,
\end{equation}
as $n\longrightarrow\infty$, uniformly in $n$. The both estimates then give the strong convergence
\begin{equation}
|v_{n}^{(2)}|^{p'-2}v_{n}^{(2)}-|v_{n}^{(1)}|^{p'-2}v_{n}^{(1)}+|w^{(1)}|^{p'-2}w^{(1)}\longrightarrow 0~~\mathrm{in}~L^{p}(\mathbb{R}^{N}),~~\mathrm{as}~n\longrightarrow\infty,
\end{equation}
and therefore,
\begin{equation}
J'(v_{n}^{(2)})=J'(v_{n}^{(1)})-J'(w^{(1)})+o(1),~~~\mathrm{as}~n\longrightarrow \infty.
\end{equation}
This implies that $(v_{n}^{(2)})_{n}$ is a (PS)-sequence for $J$ at level $d-J(w^{(1)})>0$. Thus, the nonvanishing theorem again gives the existence of $R_{1}$, $\zeta_{1}>0$ and of a sequence $(y_{n})_{n}\subset\mathbb{R}^{N}$ such that, going to a subsequence
\begin{equation}
\int_{B_{R_{1}}(y_{n})}|v_{n}^{(2)}|^{p'}\mathrm{d}x\geq \zeta_{1}>0~~~\mathrm{for~all}~n.
\end{equation}
By Lemma \ref{PS}, there is a critical point $w^{(2)}$ of $J$ such that (taking a further subsequence)
$v_{n}^{(2)}(\cdot+y_{n})\rightharpoonup w^{(2)}$ weakly and $1_{B}v_{n}^{(2)}(\cdot+y_{n})\longrightarrow 1_{B}w^{(2)}$ strongly in $L^{p'}(\mathbb{R}^{N})$, for all bounded and measurable set $B\subset\mathbb{R}^{N}$. In particular, $w^{(2)}\neq0$ and since $v_{n}^{(2)}\rightharpoonup 0$, we see that $|y_{n}|\longrightarrow\infty$ as $n\longrightarrow\infty$.

Setting $x_{n}^{(2)}=x_{n}^{(1)}=y_{n}$, we obtain $|x_{n}^{(2)}-x_{n}^{(1)}|\longrightarrow\infty$ as $n\longrightarrow\infty$, and
\begin{equation}
v_{n}-\big(w^{(1)}(\cdot+x_{n}^{(1)})+w_{n}^{(2)}(\cdot+x_{n}^{(2)})\big)=v_{n}^{(2)}(\cdot+y_{n}-x_{n}^{(2)})-w^{(2)}(\cdot-x_{n}^{(2)})\rightharpoonup0,
\end{equation}
weakly in $L^{p'}(\mathbb{R}^{N})$. In addition, the same argument as before show that
\begin{equation}
J(w^{(2)})\leq \mathop{\mathrm{lim~inf}}\limits_{n\longrightarrow\infty} J(v_{n}^{(2)})=d-J(w^{(1)})
\end{equation}
with equality if and only $v_{n}^{(2)}(\cdot+y_{n})\longrightarrow w^{(2)}$ strongly in $L^{p'}(\mathbb{R}^{N})$. If the inequality is strict, we can iterate the procedure. Since for every nontrivial critical point $w$ of $J$ we have $J(w)\geq c=\mathop{\mathrm{inf}}\limits_{\mathcal{N}}J>0$, the iterate has to stop after finitely many steps, and we obtain the desired result.
\end{proof}

\begin{Prop}\label{representation1}
Let $(\varepsilon_{n})_{n}\subset (0,\infty)$ satisfy $\varepsilon_{n}\longrightarrow0$ as $n\longrightarrow\infty$. Consider for each $n$ some $v_{n}\in\mathcal{N}_{\varepsilon_{n}}$ and assume that $J_{\varepsilon_{n}}(v_{n})\longrightarrow c_{0}$ as $n\longrightarrow\infty$. Then, there is $x_{0}\in M$, a critical point $w_{0}$ of $J_{0}$ at level $c_{0}$ and a sequence $(y_{n})_{n}\subset\mathbb{R}^{N}$ such that (up to a subsequence)
\begin{equation}
\varepsilon_{n}y_{n}\longrightarrow x_{0}~~\mathrm{and}~~||v_{n}(\cdot+y_{n})-w_{0}||_{p'}\longrightarrow 0
~~\mathrm{as}~n\longrightarrow\infty.
\end{equation}
\end{Prop}
\begin{proof}
For each $n\in\mathbb{N}$, set $v_{0,n}:=(\frac{Q_{\varepsilon_{n}}}{Q_{0}})^{\frac{1}{p}}v_{n}$. It follows that $|v_{0,n}|\leq|v_{n}|$ a.e. on $\mathbb{R}^{N}$ and that
\begin{equation}
\int_{\mathbb{R}^{N}}Q_{0}^{\frac{1}{p}}v_{0,n}\mathbf{R}^{s}\big(Q_{0}^{\frac{1}{p}}v_{0,n}\big)\mathrm{d}
=\int_{\mathbb{R}^{N}}Q_{\varepsilon_{n}}^{\frac{1}{p}}v_{n}\mathbf{R}^{s}(Q_{\varepsilon_{n}}^{\frac{1}{p}}v_{n})\mathrm{d}x>0.
\end{equation}
Therefore, setting
\begin{equation}
t_{0,n}^{2-p'}=\frac{\int_{\mathbb{R}^{N}}|v_{0,n}|^{p'}\mathrm{d}x}{\int_{\mathbb{R}^{N}}Q_{0}^{\frac{1}{p}}v_{0,n}\mathbf{R}^{s}\big(Q_{0}^{\frac{1}{p}}v_{0,n}\big)\mathrm{d}x}
\end{equation}
we find that $t_{0,n}v_{0,n}\in\mathcal{N}_{0}$ and $0<t_{0,n}\leq1$. As a consequence, we can write
\begin{equation}
\begin{aligned}
c_{0}&\leq J_{0}(t_{0,n}v_{0,n})=(\frac{1}{p'}-\frac{1}{2})t_{0,n}^{2}\int_{\mathbb{R}^{N}}Q_{0}^{\frac{1}{p}}v_{0,n}\mathbf{R}^{s}\big(Q_{0}^{\frac{1}{p}}v_{0,n}\big)\mathrm{d}x\\
&=(\frac{1}{p'}-\frac{1}{2})t_{0,n}^{2}\int_{\mathbb{R}^{N}}Q_{\varepsilon_{n}}^{\frac{1}{p}}v_{n}\mathbf{R}^{s}\big(Q_{\varepsilon_{n}}^{\frac{1}{p}}v_{n}\big)\mathrm{d}x\\
&=t_{0,n}^{2}J_{\varepsilon_{n}}(v_{n})\leq J_{\varepsilon_{n}}(v_{n})\longrightarrow c_{0}~~~\mathrm{as}~n\longrightarrow \infty.
\end{aligned}
\end{equation}
In particular, we find
\begin{equation}
\mathop{\mathrm{lim}}\limits_{n\longrightarrow\infty}~t_{0,n}=1,
\end{equation}
and $(t_{0,n}v_{0,})_{n}\subset\mathcal{N}_{0}$ is thus a minimizing sequence for $J_{0}$ on $\mathcal{N}_{0}$. Using Ekeland's variational principle and the fact that $\mathcal{N}_{0}$ is a natural constraint, we obtain the existence of a $(\mathrm{PS})_{c_{0}}$-sequence $(w_{n})_{n}\subset L^{p'}(\mathbb{R}^{N})$ for $J_{0}$ with the property that $||v_{0,n}-w_{n}||_{p'}\longrightarrow0$, as $n\longrightarrow\infty$.

By Lemma \ref{representation}, there exists a critical point $w_{0}$ for $J_{0}$ at level $c_{0}$ and a sequence $(y_{n})_{n}\subset\mathbb{R}^{N}$ such that (up to a subsequence) $||w_{n}(\cdot+y_{n})-w_{0}||_{p'}\longrightarrow 0$, as $n\longrightarrow\infty$. Therefore,
\begin{equation}
v_{0,n}(\cdot+y_{n})\longrightarrow w_{0}~~~\mathrm{strongly~in~}L^{p'}(\mathbb{R}^{N}),~~\mathrm{as}~n\longrightarrow\infty.
\end{equation}

We are going to show that $(\varepsilon_{n}y_{n})_{n}$ is bounded. Suppose not, there exist some subsequence (which we still call $(\varepsilon_{n}y_{n})_{n})$ such that $\mathop{\mathrm{lim}}\limits_{n\longrightarrow\infty}|\varepsilon_{n}y_{n}|=\infty$. We consider the following two cases.

(1)~ $Q_{\infty}=0$. In this case, by the assumption on $Q$, we have $Q(\varepsilon_{n}\cdot+\varepsilon_{n}y_{n})\longrightarrow 0$, as $n\longrightarrow\infty$, holds uniformly on bounded sets of $\mathbb{R}^{N}$. From the definition of $v_{0,n}$, we infer that $v_{0,n}(\cdot+y_{n})\rightharpoonup0$ and therefore $w_{0}=0$, in contradiction to $J_{0}(w_{0})=c_{0}>0$. Hence, $(\varepsilon_{n}y_{n})_{n}$ is bounded in this case.

(2)~$Q_{\infty}>0$. By the Fatou's lemma and the strong convergence $v_{0,n}(\cdot+y_{n})\longrightarrow w_{0}$, we deduce that
\begin{equation}
\begin{aligned}
c_{0}&=\mathop{\mathrm{lim}}\limits_{n\longrightarrow\infty}J_{\varepsilon_{n}}(v_{n})=\mathop{\mathrm{lim}}\limits_{n\longrightarrow\infty}(\frac{1}{p'}-\frac{1}{2})\int_{\mathbb{R}^{N}}|v_{n}|^{p'}\mathrm{d}x\\
&=\mathop{\mathrm{lim}}\limits_{n\longrightarrow\infty}(\frac{1}{p'}-\frac{1}{2})\int_{\mathbb{R}^{N}}|v_{n}(x+y_{n})|^{p'}\mathrm{d}x\\
&=\mathop{\mathrm{lim~inf}}\limits_{n\longrightarrow\infty}(\frac{1}{p'}-\frac{1}{2})
\int_{\mathbb{R}^{N}}\big(\frac{Q_{0}}{Q(\varepsilon_{n}x+\varepsilon_{n}y_{n})}\big)^{p'-1}|v_{0,n}(x+y_{n})|^{p'}\mathrm{d}x\\
&\geq (\frac{1}{p'}-\frac{1}{2})\int_{\mathbb{R}^{N}}\big(\frac{Q_{0}}{Q_{\infty}}\big)^{p'-1}|w_{0}|^{p'}\mathrm{d}x\\
&=\big(\frac{Q_{0}}{Q_{\infty}}\big)^{p'-1}c_{0},
\end{aligned}
\end{equation}
and this contradicts (Q1). Therefore, $(\varepsilon_{n}y_{n})_{n}$ is a bounded sequence, and we may assume (going to a subsequence) that $\varepsilon_{n}y_{n}\longrightarrow x_{0}\in\mathbb{R}^{N}$.

Since $Q(\varepsilon_{n}x+\varepsilon_{n}y_{n})\longrightarrow Q_{x_{0}}$, as $n\longrightarrow\infty$, uniformly on bounded set, the argument of case (1) above gives $Q(x_{0})>0$ and, using the Dominated Convergence Theorem, we see that $Q(x_{0})=Q_{0}$, since the following holds.
\begin{equation}
\begin{aligned}
c_{0}&=\mathop{\mathrm{lim}}\limits_{n\longrightarrow\infty}J_{\varepsilon_{n}}(v_{n})=\mathop{\mathrm{lim}}\limits_{n\longrightarrow\infty}
(\frac{1}{p'}-\frac{1}{2})\int_{\mathbb{R}^{N}}|v_{n}|^{p'}\mathrm{d}x\\
&=\mathop{\mathrm{lim}}\limits_{n\longrightarrow\infty}(\frac{1}{p'}-\frac{1}{2})\int_{\mathbb{R}^{N}}\big(\frac{Q_{0}}{Q(\varepsilon_{n}x+\varepsilon_{n}y_{n})}\big)^{p'-1}|v_{0,n}(x+y_{n})|^{p'}\mathrm{d}x\\
&=(\frac{1}{p'}-\frac{1}{2})\int_{\mathbb{R}^{N}}\big(\frac{Q_{0}}{Q(x_{0})}\big)^{p'-1}|w_{0}|^{p'}\mathrm{d}x\\
&=\big(\frac{Q_{0}}{Q_{x_{0}}}\big)^{p'-1}c_{0}.
\end{aligned}
\end{equation}
Going back to the original sequence we obtain
\begin{equation}
v_{n}(\cdot+y_{n})
=\big(\frac{Q_{0}}{Q(\varepsilon_{n}+\varepsilon_{n}y_{n})}\big)^{\frac{1}{p}}v_{0,n}(\cdot+y_{n})\longrightarrow\big(\frac{Q_{0}}{Q(x_{0})}\big)^{\frac{1}{p}}w_{0}=w_{0},~~\mathrm{as}~n\longrightarrow\infty,
\end{equation}
strongly in $L^{p'}(\mathbb{R}^{N})$, using again the Dominated Convergence Theorem. The proof is complete.
\end{proof}

%\begin{thm}
%Let $k_{0}:=\varepsilon_{0}^{-1}>0$, where $c_{0}$ is given by Proposition \ref{level}. For every %sequence $(k_{n})_{n}\subset(k_{0},\infty)$ satisfying $k_{n}\longrightarrow\infty$ as %$n\longrightarrow\infty$, and every sequence $(u_{n})_{n}$ such that $u_{n}$ is a dual ground state %of \begin{equation}
%-\Delta u-k_{n}u=Q(x)|u|^{p-2}u~~\mathrm{in}~\mathbb{R}^{N},
%\end{equation}
%there is $x_{0}\in M$, a dual ground state $u_{0}$ of (\ref{limit5}) and a sequence %$(x_{n})_{n}\subset\mathbb{R}^{N}$ such that (up to a subsequence) %$\mathop{\mathrm{lim}}\limits_{n\longrightarrow\infty}x_{n}=x_{0}$ and
%\begin{equation}
%k_{n}^{-\frac{2}{p-2}}u_{n}\big(\frac{\cdot}{k_{n}}+x_{n}\big)\longrightarrow %u_{0}~~~\mathrm{in}~L^{p}(\mathbb{R}^{N}),~~\mathrm{as}~n\longrightarrow\infty.
%\end{equation}
%\end{thm}

\begin{proof}[\bf Proof of Theorem 1.1 (ii)]
By (\ref{solution}), the dual ground state $u_{n}$ can be represented as
\begin{equation}
u_{n}(x)=k_{n}^{\frac{2s}{p-2}}\mathbf{R}^{s}\big(Q_{\varepsilon_{n}}^{\frac{1}{p}}v_{n}\big)(k_{n}x),~~x\in\mathbb{R}^{N},
\end{equation}
where $\varepsilon_{n}=k_{n}^{-1}$ and $v_{n}\in L^{p'}(\mathbb{R}^{N})$ is a least-energy critical point of $J_{\varepsilon}$, i.e., $J'_{\varepsilon}(v_{n})=0$ and $J_{\varepsilon_{n}}(v_{n})=c_{\varepsilon_{n}}$. By Lemma \ref{limit4} and Proposition \ref{representation1}, there is $x_{0}\in M$ and a sequence $(y_{n})_{n}\subset\mathbb{R}^{N}$ such that, as $n\longrightarrow\infty$, $x_{n}:=\varepsilon_{n}y_{n}\longrightarrow x_{0}$ and, going to a subsequence, $v(\cdot+y_{n})\longrightarrow w_{0}$ in $L^{p'}(\mathbb{R}^{N})$ for some least-energy critical point $w_{0}$ of $J_{0}$. Therefore, for $x\in\mathbb{R}^{N}$,
\begin{equation}
k_{n}^{-\frac{2s}{p-2}}u_{n}\big(\frac{x}{k_{n}}+x_{n}\big)=\mathbf{R}^{s}\big(Q_{\varepsilon_{n}}^{\frac{1}{p}}v_{n}\big)(x+y_{n})
=\mathbf{R}^{s}\big(Q_{\varepsilon_{n}}^{\frac{1}{p}}(\cdot+y_{n})v_{n}(\cdot+y_{n})\big)(x).
\end{equation}
On the other hand, by the continuity of $\mathbf{R}^{s}$ and the pointwise convergence $Q_{\varepsilon_{n}}(x+y_{n})\longrightarrow Q(x_{0})=Q_{0}$ as $n\longrightarrow\infty$ for all $x\in\mathbb{R}^{N}$, we have the following strong convergence
\begin{equation}
k_{n}^{-\frac{2s}{p-2}}u_{n}\big(\frac{x}{k_{n}}+x_{n}\big)\longrightarrow \mathbf{R}^{s}\big(Q_{0}^{\frac{1}{p}}w_{0}\big)~~~\mathrm{in}~L^{p}(\mathbb{R}^{N}).
\end{equation}
Setting $u_{0}=\mathbf{R}^{s}(Q_{0}^{\frac{1}{p}}w_{0})$, the properties $J_{0}(w_{0})=c_{0}$ and $J'_{0}(w_{0})=0$ imply that $u_{0}$ is a dual ground state solution of (\ref{limit5}) and this conclude the proof.
\end{proof} 

%\begin{Rem}
%(i)~The conclusion of the preceding theorem holds more generally for every sequence of dual bound %states. Indeed, in view of Proposition \ref{representation1} it is enough to have %$u_{n}(x)=k_{n}^{-\frac{2}{p-2}}\mathbf{R}^{s}(Q_{\varepsilon_{n}}^{\frac{1}{p}}v_{n})(k_{n}x)$, %where $v_{n}$ is a critical point of $J_{\varepsilon_{n}}$, and to require %$J_{\varepsilon_{n}}(v_{n})\longrightarrow c_{0}$ as $n\longrightarrow\infty$.

%(ii)~Elliptic eestimate imply that the convergence toward $u_{0}$ holds in $W^{2,q}(\mathbb{R}^{N})$ %for all $\frac{2N}{N-1}<q<\infty$. In particular, the convergence holds in %$L^{\infty}(\mathbb{R}^{N})$ and since $u_{0}\in W^{2,p}(\mathbb{R}^{N})$ we find that for every %$\delta>0$ there is $R_{\delta}>0$ such that for large $n$,
%\begin{equation}
%k_{n}^{-\frac{2}{p-2}}|u_{n}(x)|<\delta~~~~\mathrm{for~all}~|x-x_{n}|\geq\frac{R_{\delta}}{k_{n}},
%\end{equation}
%whereas $k_{n}^{-\frac{2}{p-2}}||u_{n}||_{\infty}\longrightarrow||u_{0}||_{\infty}>0$ as %$n\longrightarrow\infty$. In addition, if $\tilde{x}_{n}$ denotes any global maximum point of %$|u_{n}|$, then $\tilde{x}_{n}\longrightarrow x_{0}$ as $n\longrightarrow\infty$.
%\end{Rem}

\section*{Acknowledgements}
The research of Minbo Yang was partially supported by NSFC (11971436, 12011530199) and ZJNSF (LD19A010001).

\end{document}